\documentclass[12pt]{amsart}
\usepackage[T1]{fontenc} 
\usepackage[utf8]{inputenc}
\usepackage{tikz}

\def\gam{\gamma }
\def\lam{\lambda }
\def\gam{\gamma }

\def\RR{\mathbb R}
\def\ZZ{\mathbb Z}

\newcommand{\set}[1]{\left\{#1\right\}}%set
\providecommand{\abs}[1]{\left\lvert#1\right\rvert}
\providecommand{\norm}[1]{\left\lVert#1\right\rVert}
\DeclareMathOperator{\dist}{dist}
\DeclareMathOperator{\vect}{Vect}
\newtheorem{theorem}{Theorem}[section]
\newtheorem{proposition}[theorem]{Proposition}
\newtheorem{lemma}[theorem]{Lemma}
\newtheorem{corollary}[theorem]{Corollary}
\newtheorem*{remark}{Remark}
\newtheorem*{remarks}{Remarks}
\newtheorem*{definition}{Definition}

\numberwithin{equation}{section}
\newcommand{\remove}[1]{ }

\begin{document}
\title[Observability on small sets]{Observability of rectangular membranes and plates on small sets} 
\author{Vilmos Komornik} 
\address{Département de mathématique\\
         Université de Strasbourg\\
         7 rue René Descartes\\
         67084 Strasbourg Cedex, France}
\email{komornik@math.unistra.fr}
\author{Paola Loreti}
\address{Sapienza Università di Roma,
Dipartimento di Scienze di Base
e Applicate per l'Ingegne\-ria,
Sezione di Matematica,
via A. Scarpa n. 16,
00161 Roma, Italy}
\email{loreti@dmmm.uniroma1.it}
\thanks{Part of this work was done during the visit of the second author to the Department of Mathematics of the University of Strasbourg in January 2013. The visit was partially supported by INdAM, GDRE ConEDP Project.}
\date{Version of 2013-08-21-a}
%\dedicatory{}

\begin{abstract}
Since the works of Haraux and Jaffard we know that rectangular plates may be observed by subregions not satisfying the geometrical control condition. We improve these results by observing only on an arbitrarily short segment inside the domain. The estimates may be strengthened by observing on several well-chosen segments.

In the second part of the paper we establish various observability theorems for rectangular membranes by applying Mehrenberger's recent generalization of Ingham's theorem.
\end{abstract}
\maketitle

\section{Introduction and statement of the main result}\label{s1}

Given an open rectangle
\begin{equation*}
\Omega=(0,\ell_1)\times (0,\ell_2)
\end{equation*}
with boundary $\Gamma$, we consider solutions of the following system:
\begin{equation}\label{11}
\begin{cases}
u''+\Delta^2 u=0\quad &\text{in}\quad \RR\times\Omega,\\
u=\Delta u=0 &\text{on}\quad\RR\times\Gamma,\\
u(0)=u_0 &\text{in}\quad\Omega,\\
u'(0)=u_1 &\text{in}\quad\Omega.
\end{cases}
\end{equation}
They describe the small transversal vibrations of a hinged plate. We are interested in the observability of this system on small subsets of $\Omega$. 

In order to recall two results, it is convenient to introduce the Hilbert spaces $D^s$ for $s\in\RR$ as follows. We fix an orthonormal basis $(e_k)$ in $L^2(\Omega)$, formed by eigenfunctions of $-\Delta$ in $H^1_0(\Omega)$ with corresponding eigenvalues $\lam_k>0$, and we denote by $D^s$ the completion of $\vect\set{e_k}$ with respect to the Euclidean norm
\begin{equation*}
\norm{\sum_{k=1}^{\infty}c_ke_k}_s:=\left( \sum_{k=1}^{\infty}\lam_k^{s}\abs{c_k}^2\right)^{1/2}.
\end{equation*}
Identifying $L^2(\Omega)$ with its dual, we have in particular
\begin{multline*}
D^0=L^2(\Omega),\quad D^1=H^1_0(\Omega),\quad D^{-1}=H^{-1}(\Omega)\\
\quad\text{and}\quad 
D^2=H^2(\Omega)\cap H^1_0(\Omega)
\end{multline*}
(with equivalent norms).

\begin{figure}
\centering
\begin{minipage}[b]{0.4\textwidth}
\scalebox{0.8}{\begin{tikzpicture}
\draw[thick] (0,0)--(0,4)--(6,4)--(6,0)--(0,0);
\draw[ultra thick,red] (1,0)--(1,4);
\draw[ultra thick,red] (2,0)--(2,4);
\draw[ultra thick,red] (3,0)--(3,4);
\node [below left] at (0,0) {$0$};
\node [below] at (1,0) {$\alpha_1$};
\node [below] at (2,0) {$\alpha_2$};
\node [below] at (3,0) {$\alpha_3$};
\node [below right] at (6,0) {$\ell_1$};
\node [above left] at (0,4) {$\ell_2$};
\end{tikzpicture}}
\caption{}%\ref{t12}
\label{f2}
\end{minipage}%
\hspace{0.04\textwidth}%
\begin{minipage}[b]{0.4\textwidth}
\scalebox{0.8}{\begin{tikzpicture}
\draw[thick] (0,0)--(0,4)--(6,4)--(6,0)--(0,0);
\draw[ultra thick,red] (1,2.5)--(1,3);
\draw[ultra thick,red] (2,0.5)--(2,1);
\draw[ultra thick,red] (3,2)--(3,3);
\node [below left] at (0,0) {$0$};
\node [below] at (1,0) {$\alpha_1$};
\node [below] at (2,0) {$\alpha_2$};
\node [below] at (3,0) {$\alpha_3$};
\node [below right] at (6,0) {$\ell_1$};
\node [above left] at (0,4) {$\ell_2$};
\end{tikzpicture}}
\caption{}%\ref{t13}
\label{f3}
\end{minipage}
\end{figure}

Improving a theorem of Haraux \cite{Har1989}, Jaffard \cite{Jaf1988}, \cite{Jaf1990} proved the following
\begin{theorem}\label{t11}
If $\omega$ is an arbitrary non-empty open subset of $\Omega$ and $T>0$ an arbitrarily small number, then  the solutions of \eqref{11} satisfy the estimates 
\begin{equation}\label{12}
\norm{u_0}_{0}^2+\norm{u_1}_{-2}^2\asymp \int_0^T\int_{\omega}\abs{u}^2\ dx\ dt
\end{equation}
for all $(u_0,u_1)\in D^0\times D^{-2}$.
\end{theorem}
\remove{See Figure \ref{f1}.}

Here and in the sequel the notation $A\asymp B$ means that $c_1 A\le B\le c_2A$ with two positive constants $c_1, c_2$ which do not depend on the particular initial data or on the particular choice of some parameters $a_k$.

Next we recall a theorem from \cite{KomLor2011-152}, where we observe on a finite number of translates of the vertical side of $\Omega$ (see Figure \ref{f2}):

\begin{theorem}\label{t12}
Choose finitely many points $\alpha_{1},\ldots, \alpha_{M}$ in $(0,\ell_1)$ such that
\begin{equation*}
\alpha_{1},\ldots, \alpha_{M}, \ell_1
\end{equation*}
are linearly independent over the field  of rational numbers and that they belong to  a real algebraic field of degree $M+1$. 

Then for every positive number $T$ there exists a constant $c$ such that, setting $s:=-1/M$ for brevity, the solutions of \eqref{11}
satisfy the estimates
\begin{equation}\label{13}
\norm{u_0}_{s}^2+\norm{u_1}_{s-2}^2 \le c\sum_{j=1}^{M}\int_{0}^{T}\int_{0}^{\ell_2} \abs{u(t,\alpha_j,x_2)}^{2}\ dx_2\ dt
\end{equation}
for all $(u_0,u_1)\in D^1\times D^{-1}$.
\end{theorem}

Here and in the sequel the letter $c$ denotes diverse positive constants which do not depend on the particular initial data or on the particular choice of some parameters $a_k$.

Observe that when $M\to\infty$, then the norm on the left side converges to the left side of \eqref{12}.

The first result of the present paper is the following improvement of Theorem \ref{12} where we observe only on a finite number of small segments (see Figure \ref{f3}): 

\begin{theorem}\label{t13}
Let $\alpha_{1},\ldots, \alpha_{M}$ be as in Theorem \ref{t12} and set $s:=-1/M$ again. Furthermore, fix $M$ arbitrary non-empty open subintervals $I_1,\ldots, I_M$ of $(0,\ell_2)$.

Then for every positive number $T$ there exists a constant $c$ such that the solutions of  \eqref{11} satisfy the estimates
\begin{equation}\label{14}
\norm{u_0}_{s}^2+\norm{u_1}_{s-2}^2\le c\sum_{j=1}^{M}\int_{0}^{T}\int_{I_j} \abs{u(t,\alpha_j,x_2)}^{2}\ dx_2\ dt
\end{equation}
for all $(u_0,u_1)\in D^1\times D^{-1}$.
\end{theorem}

In the second half of the paper we are investigating the observability of rectangular membranes, modelised by the following system with the same rectangle $\Omega$ as above:
\begin{equation}\label{15}
\begin{cases}
u''-\Delta u=0\quad &\text{in}\quad \RR\times\Omega,\\
u=0 &\text{on}\quad\RR\times\Gamma,\\
u(0)=u_0 &\text{in}\quad\Omega,\\
u'(0)=u_1 &\text{in}\quad\Omega.
\end{cases}
\end{equation}

We start again by recalling some earlier results. The first one if a very special case of a general theorem of L. F. Ho \cite{Ho1986} and J.-L. Lions, \cite{Lio1988}, obtained by the multiplier method. Setting 
\begin{equation*}
\Gamma_0:=\left[ \set{0}\times (0,\ell_2)\right]\cup \left[ (0,\ell_1)\times\set{0}\right]
\end{equation*}
(see Figure \ref{f4}) we have the

\begin{theorem}\label{t14}
If $T>0$ is sufficiently large, then  the solutions of \eqref{15} satisfy the estimates
\begin{equation}\label{16}
\norm{u_0}_1^2+\norm{u_1}_0^2\asymp \int_0^T\int_{\Gamma_0}\abs{\partial_{\nu}u}^2\ d\Gamma\ dt
\end{equation} for all $(u_0,u_1)\in D^1\times D^0$.
\end{theorem}

\begin{figure}
\centering
\begin{minipage}[b]{0.4\textwidth}
\begin{tikzpicture}
\draw[thick] (0,0)--(0,4)--(4,4)--(4,0);
\draw[line width=1mm,red] (0,0)--(4,0);
\draw[line width=1mm,red] (0,0)--(0,4);
\node [below left] at (0,0) {$0$};
\node [below right] at (4,0) {$\ell_1$};
\node [above left] at (0,4) {$\ell_2$};
\end{tikzpicture}
\caption{}%Th. \ref{t14}
\label{f4}
\end{minipage}%
\hspace{0.04\textwidth}%
\begin{minipage}[b]{0.4\textwidth}
\begin{tikzpicture}
\draw[thick] (0,0)--(0,4)--(4,4)--(4,0)--(0,0);
\fill[red] (1.5,0) rectangle (2,4);
\fill[red] (0,1) rectangle (4,2);
\node [below left] at (0,0) {$0$};
\node [below] at (1.5,0) {$a$};
\node [below] at (2,0) {$b$};
\node [left] at (0,1) {$c$};
\node [left] at (0,2) {$d$};
\node [below right] at (4,0) {$\ell_1$};
\node [above left] at (0,4) {$\ell_2$};
\end{tikzpicture}
\caption{}%\eqref{111}
\label{f5}
\end{minipage}
\end{figure}

The proof of Ho and Lions provides the optimal time estimate: $T$ has to be larger than the diagonal of the rectangle. See Bardos, Lebeau and Rauch \cite{BarLebRau1992} for a very general condition on $T$ based on microlocal analysis.

Mehrenberger \cite{Meh2009} has found an elementary proof of Theorem \ref{t14} (under a stronger assumption on $T$) by a clever application of an Ingham type theorem. Adapting this method, the following theorem was obtained in \cite{KomMia2013} (see Figure \ref{f5}):

\begin{theorem}\label{t15}
Fix two non-empty subintervals $(a,b)\subset (0,\ell_1)$ and $(c,d)\subset (0,\ell_2)$ and set 
\begin{equation*}
S:=\left[ (a,b)\times (0,\ell_2)\right]\cup \left[ (0,\ell_1)\times (c,d)\right].
\end{equation*}

If $T>0$ is sufficiently large, then the solutions of \eqref{15} satisfy the estimates
\begin{equation}\label{17}
\norm{u_0}_1^2+\norm{u_1}_0^2\le c \int_0^T\int_S \abs{u'(t,x,y)}^2\ dx\ dy\ dt
\end{equation}
for all $(u_0,u_1)\in D^1\times D^0$.
\end{theorem}

\begin{figure}
\centering
\begin{minipage}[b]{0.4\textwidth}
\begin{tikzpicture}
\draw[thick] (0,0)--(0,4)--(4,4)--(4,0);
\draw[line width=1mm,red] (0,0)--(4,0);
\fill[red] (1,0) rectangle (2,4);
\node [below left] at (0,0) {$0$};
\node [below] at (1,0) {$a$};
\node [below] at (2,0) {$b$};
\node [below right] at (4,0) {$\pi$};
\node [above left] at (0,4) {$\pi$};
\end{tikzpicture}
\caption{}%\eqref{19}
\label{f6}
\end{minipage}%
\hspace{0.04\textwidth}%
\begin{minipage}[b]{0.4\textwidth}
\begin{tikzpicture}
\draw[thick] (0,0)--(0,4)--(4,4)--(4,0)--(0,0);
\draw[line width=0.5mm,red] (1.5,0)--(1.5,4);
\fill[red] (0,1) rectangle (4,2);
\node [below left] at (0,0) {$0$};
\node [left] at (0,1) {$c$};
\node [left] at (0,2) {$d$};
\node [below] at (1.5,0) {$\alpha$};
\node [below right] at (4,0) {$\pi$};
\node [above left] at (0,4) {$\pi$};
\end{tikzpicture}
\caption{}
\label{f7}
\end{minipage}
\end{figure}

\begin{remark}
The direct inequality 
\begin{equation*}
\int_0^T\int_S \abs{u'(t,x,y)}^2\ dx\ dy\ dt\le c\left( \norm{u_0}_1^2+\norm{u_1}_0^2\right)
\end{equation*}
also holds by a short ingenious argument of Haraux \cite{Har1988} (see also \cite{Lio1988b} I, Lemme 2.1, pp. 401--403).
\end{remark}

We give three further results in this direction by showing that we may combine the boundary and internal observations, and that we may restrict the internal observation onto a straight line for a large set of initial data. 

For simplicity of notations we consider only the square domain $\Omega=(0,\pi)\times (0,\pi)$, but the results extend to arbitrary rectangular domains by obvious adaptations of the proofs.

Our first result is the following (see Figure \ref{f6}):

\begin{theorem}\label{t16}
Fix $0\le a<b\le\pi$ arbitrarily. If $T>0$ is sufficiently large, then the  solutions of \eqref{15} satisfy the estimates
\begin{multline}\label{18}
\norm{u_0}_1^2+\norm{u_1}_0^2\le c\left( \int_0^T\int_0^{\pi}\abs{\partial_{\nu}(t,x_1,0)}^2 \ dx_1\ dt\right.\\
\left. +\int_0^T\int_a^b\int_0^{\pi}\abs{u'(t,x_1,x_2)}^2\ dx_2\ dx_1\ dt\right).
\end{multline}
\end{theorem}

\begin{remark}
Here too, the direct inequality 
\begin{multline*}
\int_0^T\int_0^{\pi}\abs{\partial_{\nu}(t,x_1,0)}^2 \ dx_1\ dt\\
+\int_0^T\int_a^b\int_0^{\pi}\abs{u'(t,x_1,x_2)}^2\ dx_2\ dx_1\ dt 
\le c'\left( \norm{u_0}_1^2+\norm{u_1}_0^2\right)
\end{multline*}
also holds by combining the direct part of Theorem \ref{t14} with Haraux's estimate mentioned above.
\end{remark}

\begin{figure}
\centering
\begin{minipage}[b]{0.4\textwidth}
\begin{tikzpicture}
\draw[thick] (0,0)--(0,4)--(4,4)--(4,0)--(0,0);
\draw[ultra thick,red] (1.5,0)--(1.5,4);
\draw[line width=1mm,red] (0,0)--(4,0);
\node [below left] at (0,0) {$0$};
\node [below] at (1.5,0) {$\alpha$};
\node [below right] at (4,0) {$\pi$};
\node [above left] at (0,4) {$\pi$};
\end{tikzpicture}
\caption{}
\label{f8}
\end{minipage}%
\hspace{0.04\textwidth}%
\begin{minipage}[b]{0.4\textwidth}
\begin{tikzpicture}
\draw[thick] (0,0)--(0,4)--(4,4)--(4,0)--(0,0);
\draw[ultra thick,red] (0,3)--(4,3);
\draw[ultra thick,red] (1.5,0)--(1.5,4);
\node [below left] at (0,0) {$0$};
\node [below] at (1.5,0) {$\alpha$};
\node [left] at (0,3) {$\beta$};
\node [below right] at (4,0) {$\pi$};
\node [above left] at (0,4) {$\pi$};
\end{tikzpicture}
\caption{}
\label{f9}
\end{minipage}
\end{figure}

In our last two theorems we consider only \emph{smooth solutions}, corresponding to initial data $(u_0,u_1)$ belonging to $D^2\times D^1$.

\begin{definition}
A function $f\in L^1(0,\pi)$ is \emph{$p$-symmetric} for some integer $p\ge 2$ if its $2\pi$-periodic odd extension satisfies the equality
\begin{equation*}
\sum_{k=1}^pf\left( t+\frac{2k\pi}{p}\right)=0
\end{equation*}
for almost every $t\in\RR$.
\end{definition}

Now fix a real number $\alpha\in (0,\pi)$ such that $\alpha/\pi$ is rational, and let $p$ be the smallest positive integer for which $p\alpha/\pi$ is integer. Observe that $\sin n\alpha =0$ if and only if $n$ is an integer multiple of $p$.

\begin{theorem}\label{t17}
If $T>0$ is sufficiently large, then there exists a positive constant $c$ such that all smooth solutions of \eqref{15}, corresponding to $p$-symmetric (in $x_1$) initial data $u_0, u_1$,  satisfy the estimates
\begin{multline}\label{19}
\norm{u_0}_1^2+\norm{u_1}_0^2\le c 
\int_0^T\int_0^{\pi} \abs{u'\left(t,\alpha,x_2\right)}^2\ dx_2 \ dt\\
+c\int_0^T\int_0^{\pi}\int_c^d\abs{u'(t,x_1,x_2)}^2\ dx_2\ dx_1\ dt
\end{multline}
and
\begin{multline}\label{110}
\norm{u_0}_1^2+\norm{u_1}_0^2\le c 
\int_0^T\int_0^{\pi} \abs{u'\left(t,\alpha,x_2\right)}^2\ dx_2 \ dt\\
+c\int_0^T\int_0^{\pi}\abs{\partial_{\nu}(t,x_1,0)}^2 \ dx_1\ dt.
\end{multline}
\end{theorem}
See Figures \ref{f7} and \ref{f8}.

\begin{remark}
For $p=2$ the solutions are odd in $x_1$ with respect to $\pi/2$. Hence the theorem also follows by applying Mehrenberger's method directly to the left half of $\Omega$ with homogeneous Neumann boundary condition on the right side and homogeneous Dirichlet boundary conditions on the three other sides. The case $p\ge 3$ does not seem to follow from earlier results.
\end{remark}

Finally fix two real numbers $\alpha,\beta\in (0,\pi)$ such that $\alpha/\pi$ and $\beta/\pi$ are rational, and let $p,q$ be the smallest positive integers for which $p\alpha/\pi$ and $q\beta/\pi$ are integers.
In the following theorem we consider only $(p,q)$-symmetric initial initial data, i.e., initial data which are $p$-symmetric in $x_1$ and $q$-symmetric in $x_2$.

\begin{theorem}\label{t18}
If $T>0$ is sufficiently large, then there exists a positive constant $c$ such that all smooth solutions of \eqref{15} corresponding to $(p,q)$-symmetric initial data satisfy the estimates
\begin{multline}\label{111}
\int_0^T\int_0^{\pi} \abs{u'\left(t,\alpha,x_2\right)}^2\ dx_2 \ dt
+\int_0^T\int_0^{\pi} \abs{u'\left(t,x_1, \beta\right)}^2\ dx_1 \ dt\\
\ge c\int_{\Omega}\abs{\nabla u_0}^2+\abs{u_1}^2\ dx.
\end{multline}
\end{theorem}
See Figure \ref{f9}.

\begin{remark}
Similarly to the preceding result, for $p=q=2$ the theorem also follows by applying Mehrenberger's method directly to the subsquare $(0,\pi/2)\times(0,\pi/2)$ with homogeneous Dirichlet boundary conditions on the upper and left sides, and homogeneous Neumann boundary conditions on the lower and right sides, because the solutions are odd with respect to $\pi/2$ in both variables. The other cases of Theorem \ref{t18} do not seem to follow from earlier results.
\end{remark}

\section{Proof of Theorem \ref{t13}}\label{s2}

We begin by recalling several results from analysis and Diophantine approximation theory. We fix three positive numbers $u,v,z$, and for $k=(k_1,k_2)\in\ZZ^2$ we write for brevity $\abs{k}^2=uk_1^2+vk_2^2$.

\begin{theorem}\label{t21}
(See \cite[Theorem 4.2]{TenTuc}.) 
Given a non-empty bounded open set $\omega$ in $\RR^2$, there exist two positive constants $c_1, c_2$ such that
\begin{equation*}
\int_{\omega}\abs{\sum_{k_1,k_2=1}^{\infty}a_ke^{i(zk_2x_2+\abs{k}^2t)}}\asymp \sum_{k_1,k_2=1}^{\infty}\abs{a_k}^2
\end{equation*}
for all square summable families $(a_k)$ of complex numbers.
\end{theorem}

\begin{remarks}\mbox{}

\begin{itemize}
\item The theorem was stated in \cite{TenTuc} with $z=1$ but the case $z\ne 1$ hence follows by a change of variable.
\remove{\item There is a misprint in the statement of \cite[Theorem 4.2]{TenTuc} because $m$ should run only over the positive integers. Indeed, the inequality fails for example for all two-term sums with $a_{-m,n}=-a_{mn}\ne 0$.}
\item In \cite{TenTuc} only the so-called \emph{inverse inequality}
\begin{equation*}
\sum_{k_1,k_2=1}^{\infty}\abs{a_k}^2\le c \int_{\omega}\abs{\sum_{k_1,k_2=1}^{\infty}a_ke^{i(zk_2x_2+\abs{k}^2t)}}
\end{equation*}
 was stated. However, the inverse inequality implies that the set of exponents has a uniform gap (see \cite[pp. 60, 154]{KomLorbook}, and the existence of a uniform gap implies the so-called \emph{direct inequality}
\begin{equation*}
\int_{\omega}\abs{\sum_{k_1,k_2=1}^{\infty}a_ke^{i(zk_2x_2+\abs{k}^2t)}}\le c \sum_{k_1,k_2=1}^{\infty}\abs{a_k}^2
\end{equation*}
(see \cite[p. 155]{KomLorbook}: the proof of the third step is valid for every bounded integration domain).
\end{itemize}
\end{remarks}

In the following proposition we denote by $B_R$ the open ball of radius $R$ in $\RR^N$, centered at the origin.

\begin{proposition}\label{p22}
(See \cite[Proposition 8.4, p. 161]{KomLorbook}) Let $(\omega_k)_{k\in K}$ be a family  of vectors in $\RR^N$, satisfying the uniform gap condition: 
\begin{equation*}
\inf\set{\abs{\omega_k-\omega_n}\ :\ k\ne n}>0.
\end{equation*}
Let 
$K=K_1\cup\dots\cup K_m$ be a partition of $K$, and assume that the estimates
\begin{equation*}
\int _{B_{R_j}}\Bigl\vert \sum_{k\in K_j}x_ke^{i\omega_k\cdot t}\Bigr\vert^2\ dt
\asymp\sum_{k\in K_j}\abs{x_k}^2,\quad j=1,\dots, m
\end{equation*}
hold  for all finite sums with complex coefficients $x_k$, with suitable
numbers $R_j$. Then we also have 
\begin{equation*}
\int _{B_R}\Bigl\vert \sum_{k\in K}x_ke^{i\omega_k\cdot t}\Bigr\vert^2\ dt
\asymp\sum_{k\in K}\abs{x_k}^2
\end{equation*}
for every $R>R_1+\dots+R_m$, for all sums with square-summable complex coefficients.
\end{proposition}

\begin{lemma}\label{l23}
(See \cite[Lemma 2.3]{KomLor2011-152}) We have
\begin{equation*}
\abs{\sin kx}\ge {2}\dist\left( \frac{kx}{\pi},\ZZ\right) 
\end{equation*}
for all real numbers $x$ and integers $k$.
\end{lemma}

\begin{theorem}\label{t24}
(See \cite[p. 74]{Cas}.) If the numbers $\theta_{1}$,\dots, $\theta_{M}$ and $1$ are linearly
independent over the field of rational numbers  and they belong to  a real algebraic field of degree $M+1$, then there exists a positive
constant $\gam$ such that
\begin{equation*}
\max_{1\le j\le M}\dist\left( k\theta_j,\ZZ\right) \ge\gam k^{-1/M}
\end{equation*}
for all positive integers $k$.
\end{theorem}

Now we turn to the proof of Theorem \ref{t13}. First we observe that Theorem \ref{t21} remains valid by simple changes of variables if we change the expression $k_2x_2+\abs{k}^2t$ to $-k_2x_2+\abs{k}^2t$, $k_2x_2-\abs{k}^2t$ or to $-k_2x_2-\abs{k}^2t$. 

Next we observe that under the assumptions of Theorem \ref{t21} we also have the more general estimate
\begin{equation}\label{21}
\int_{\omega}\abs{f(t,x_2)}^2\asymp \sum_{k_1,k_2=1}^{\infty}\left( \abs{a_k}^2+\abs{b_k}^2+\abs{c_k}^2+\abs{d_k}^2\right) 
\end{equation} 
for all square summable families $(a_k)$,  $(b_k)$, $(c_k)$, $(d_k)$ of complex numbers, where we use the notation
\begin{multline}\label{22}
f(t,x_2):=\sum_{k_1,k_2=1}^{\infty}\left( a_ke^{i(zk_2x_2+\abs{k}^2t)}+ b_ke^{i(-zk_2x_2+\abs{k}^2t)}\right. \\
+\left. c_ke^{i(zk_2x_2-\abs{k}^2t)}+d_ke^{i(-zk_2x_2-\abs{k}^2t)}\right) .
\end{multline}
Indeed, applying Proposition \ref{p22} with $m=4$, 
\begin{align*}
\set{\omega_k\ :\ k\in K_1}&=\set{(zk_2,\abs{k}^2)\ :\ k_1,k_2=1,2,\ldots},\\
\set{\omega_k\ :\ k\in K_2}&=\set{(-zk_2,\abs{k}^2)\ :\ k_1,k_2=1,2,\ldots},\\
\set{\omega_k\ :\ k\in K_3}&=\set{(zk_2,-\abs{k}^2)\ :\ k_1,k_2=1,2,\ldots},\\
\set{\omega_k\ :\ k\in K_4}&=\set{(-zk_2,-\abs{k}^2)\ :\ k_1,k_2=1,2,\ldots}
\end{align*}
and with arbitrary positive numbers $R_j$ (this is possible in view of our first observation and because the union of exponents clearly has a uniform gap), we conclude that the estimates \eqref{21} hold for every disc centered at the origin. By translation invariance the estimates then hold for all discs (of arbitrary  center), and then they also hold for every non-empty bounded open set $\omega$ because there exist two discs $B_1$ and $B_2$ such that $B_1\subset\omega\subset B_2$.

For rectangular domains the solution of \eqref{11} has the following form:
\begin{equation}
\label{23}
u(t,x)
= \sum _{k_1,k_2=1}^\infty   \left (a_ke^{i(\frac{k_1^2\pi^2}{\ell_1^2}+\frac{k_2^2\pi^2}{\ell_2^2})t}+b_ke^{-i(\frac{k_1^2\pi^2}{\ell_1^2}+\frac{k_2^2\pi^2}{\ell_2^2})t}\right )\prod_{j=1}^2\sin \frac{k_jx_j\pi}{\ell_j}.
\end{equation} 
We note that the eigenvalue corresponding to the eigenfunction
\begin{equation*}
e_k(x)=\sin \frac{k_1x_1\pi}{\ell_1}\sin \frac{k_2x_2\pi}{\ell_2}
\end{equation*}
is equal to
\begin{equation*}
\lam_k=\frac{k_1^2\pi^2}{\ell_1^2}+\frac{k_2^2\pi^2}{\ell_2^2}.
\end{equation*}

Using formula \eqref{23} and the definition of the $s$-norms, a direct computation yields the well-known energy estimate
\begin{equation}\label{24}
\norm{u_0}_{s}^2+\norm{u_1}_{s-2}^2\asymp \sum _{k_1,k_2=1}^\infty \lam_k^{s}\left( \abs{a_k}^2+\abs{b_k}^2\right).
\end{equation}

Next, setting $z=\pi/\ell_2$, $u=\pi^2/\ell_1^2$, $v=\pi^2/\ell_2^2$ and 
\begin{equation*}
 \abs{k}^2=uk_1^2+vk_2^2 (=\lam_k),
\end{equation*}
we may rewrite \eqref{23} as follows:
\begin{multline*}
u(t,x):=\frac{1}{2i}\sum_{k_1,k_2=1}^{\infty}\left( a_k\sin\frac{k_1x_1\pi}{\ell_1}\right) \left( e^{i(zk_2x_2+\abs{k}^2t)}-e^{i(-zk_2x_2+\abs{k}^2t)}\right) \\
+\left( b_k\sin\frac{k_1x_1\pi}{\ell_1}\right) \left( e^{i(zk_2x_2-\abs{k}^2t)}-e^{i(-zk_2x_2-\abs{k}^2t)}\right) .
\end{multline*}
Applying the estimate \eqref{21} for each $j=1,\ldots, M$ for the function $f(t,x_2):=u(t,\alpha_j,x_2)$, we obtain that
\begin{equation*}
\int_0^T\int_{I_j}\abs{u(t,\alpha_j,x_2)}^2\ dx_2\ dt\asymp \sum_{k_1,k_2=1}^{\infty}\left( \abs{a_k}^2+\abs{b_k}^2\right)\sin^2\frac{k_1\alpha_j\pi}{\ell_1} .
\end{equation*}
Summarizing for $j$ we get
\begin{equation*}
\sum_{j=1}^M\int_0^T\int_{I_j}\abs{u(t,\alpha_j,x_2)}^2\ dx_2\ dt\asymp \sum_{k_1,k_2=1}^{\infty}\left( \abs{a_k}^2+\abs{b_k}^2\right)\sum_{j=1}^M\sin^2\frac{k_1\alpha_j\pi}{\ell_1}.
\end{equation*}
In view of \eqref{24} the theorem will follows if we show the estimate
\begin{equation*}
\sum_{j=1}^M\sin^2\frac{k_1\alpha_j\pi}{\ell_1}\ge c \lam_k^{-1/M}.
\end{equation*}
This follows from Lemma \ref{l23} and Theorem \ref{t24} as in \cite{KomLor2011-152}:
\begin{multline*}
\sum_{j=1}^M\sin^2\frac{k_1\alpha_j\pi}{\ell_1}\ge 4\sum_{j=1}^M\dist\left( \frac{k_1\alpha_j}{\ell_1},\ZZ\right)^2 \ge 
4\gam^2 k_1^{-2/M}\\
\ge 4\gam^2 \left( k_1^2+k_2^2\right) ^{-1/M}\asymp \lam_k^{-1/M}.
\end{multline*}

\section{Proof of Theorems \ref{t15} and \ref{t16}}\label{s3}

For the reader's convenience we also reprove Theorem \ref{t15} by providing a slightly better condition on $T$.

Ingham's theorem \cite{Ing1936} and its various generalizations are very useful in control theory; see, e.g., \cite{BaiKomLor103}, \cite{BaiKomLor111}, \cite{KomLorbook}, \cite{Lor2004}, \cite{LorMeh}, \cite{LorVal}. Our proofs are based on a recent  generalization  that we recall now. Let $(\omega_k)_{k=-\infty}^{\infty}$ be a sequence of real numbers, satisfying for some nonnegative integer $n$ and for some positive real number $\gamma$ the following \emph{partial gap condition}:
\begin{equation}\label{31}
\abs{\omega_{k'}-\omega_k}\ge\abs{k'-k}\gamma\quad\text{whenever}\quad \max\set{\abs{k'},\abs{k}}\ge n.
\end{equation} 

\begin{theorem}\label{t31}
(Mehrenberger \cite{Meh2009}) Assume \eqref{31} and let $T>\frac{2\pi}{\gamma}$. Then the following inequality holds for all square summable sequences $(a_k)_{k=-\infty}^{\infty}$ of complex numbers:
\begin{multline}
\label{32}
\int_0^T\abs{\sum_{k=-\infty}^{\infty} a_ke^{i\omega_kt}}^2\ dt
\ge \frac{2T}{\pi}\left( \sum_{\abs{k}\ge n}\abs{a_k}^2 - \left( \frac{2\pi}{T\gamma}\right)^2\sum_{k=-\infty}^{\infty}\abs{a_k}^2\right).
\end{multline} 
\end{theorem}

We also recall from Mehrenberger \cite{Meh2009} the following lemma:

\begin{lemma}\label{l32}
If $k_1, k_1', k_2$ are positive integers satisfying  $\max\set{k_1,k_1'}\ge k_2$, then
\begin{equation*}
\abs{\sqrt{k_1^2+k_2^2}-\sqrt{(k_1')^2+k_2^2}}\ge\frac{\abs{k_1-k_1'}}{2\sqrt{2}}.
\end{equation*}
Furthermore, we have also the trivial inequality
\begin{equation*}
\sqrt{1+k_2^2}-\left( -\sqrt{1+k_2^2}\right)\ge\frac{1}{2\sqrt{2}}.
\end{equation*}
\end{lemma}

Combining these two results we obtain the following corollary where we still use the notations
\begin{equation*}
k=(k_1,k_2)\quad\text{and}\quad \abs{k}:=\sqrt{k_1^2+k_2^2}:
\end{equation*}

\begin{corollary}\label{c33}
If $T>4\sqrt{2}\pi$, then for any fixed positive integer $k_2$ we have the following inequality for all square summable sequences $(a_k)$ and $(b_k)$ of complex numbers:
\begin{multline}\label{33}
\int_0^T\abs{\sum_{k_1=1}^{\infty}a_ke^{i\abs{k}t}+b_ke^{-i\abs{k}t}}^2\ dt\\
\ge \left( \frac{2T}{\pi}-\frac{64\pi}{T}\right)\sum_{k_1\ge k_2} \left( \abs{a_k}^2+\abs{b_k}^2\right)
-\frac{64\pi}{T}\sum_{k_1<k_2} \left( \abs{a_k}^2+\abs{b_k}^2\right).
\end{multline}
\end{corollary}

Given any interval $(a,b)\subset (0,\pi)$ we set
\begin{equation}\label{34}
m_{a,b}:=\inf_n\int_a^b\sin^2 (ny)\ dy,
\end{equation} 
where $n$ runs over the positive integers. Then we have (see, e.g., \cite[p. 1548]{KomLor2011-152})
\begin{equation}\label{35}
0<m_{a,b}\le \int_a^b \sin^2 (ny)\ dy\le\frac{\pi}{2},
\quad n=1,2,\ldots 
\end{equation} 

Now we are ready for the proof of Theorem \ref{t15}.  The solutions are now given by the formulas

\begin{equation}\label{36}
u(t,x_1,x_2)=\sum_{k_1=1}^{\infty}\sum_{k_2=1}^{\infty}\left( a_ke^{i\abs{k}t}+b_{k}e^{-i\abs{k}t}\right)\sin (k_1x_1)\sin (k_2x_2),
\end{equation} 
where the complex coefficients satisfy the equality
\begin{equation}\label{37}
\frac{\pi^2}{2}\sum_{k_1=1}^{\infty}\sum_{k_2=1}^{\infty}\abs{k}^2\left( \abs{a_k}^2+\abs{b_k}^2\right)=\int_{\Omega}\abs{\nabla u_0}^2+\abs{u_1}^2\ dx.
\end{equation}

\begin{proof}[Proof of the estimate \eqref{17}]
Using the representation \eqref{36} of the solution and applying the preceding corollary  we obtain for $T>4\sqrt{2}\pi$ the following inequalities:
\begin{align*}
\int_0^T\int_0^{\pi}&\abs{u'(t,x_1,x_2)}^2\ dx_2\ dt\\
&=\int_0^T\int_0^{\pi}\abs{\sum_{k_1=1}^{\infty}\sum_{k_2=1}^{\infty}i\abs{k}\left( a_ke^{i\abs{k}t}-b_{k}e^{-i\abs{k}t}\right)\sin (k_1x_1)\sin (k_2x_2)}^2\ dx_2\ dt\\
&=\frac{\pi}{2}\sum_{k_2=1}^{\infty}\int_0^T\abs{\sum_{k_1=1}^{\infty}i\abs{k}\left( a_ke^{i\abs{k}t}-b_{k}e^{-i\abs{k}t}\right)\sin (k_1x_1)}^2\ dx_2\ dt\\
&\ge \frac{\pi}{2}\sum_{k_2=1}^{\infty}\left( \frac{2T}{\pi}-\frac{64\pi}{T}\right)\sum_{k_1=k_2}^{\infty}\abs{k}^2\left( \abs{a_k}^2 +\abs{b_k}^2 \right)\sin^2 (k_1x_1)\\
&\qquad\qquad -\frac{\pi}{2}\sum_{k_2=1}^{\infty}\frac{64\pi}{T}\sum_{k_1=1}^{k_2-1}\abs{k}^2\left( \abs{a_k}^2 +\abs{b_k}^2 \right)\sin^2 (k_1x_1)\\
&= \left( T-\frac{32\pi^2}{T}\right)\sum_{k_2=1}^{\infty}\sum_{k_1=k_2}^{\infty}\abs{k}^2\left( \abs{a_k}^2 +\abs{b_k}^2 \right)\sin^2 (k_1x_1)\\
&\qquad\qquad -\frac{32\pi^2}{T}\sum_{k_2=1}^{\infty}\sum_{k_1=1}^{k_2-1}\abs{k}^2\left( \abs{a_k}^2 +\abs{b_k}^2 \right)\sin^2 (k_1x_1).
\end{align*}
(In the last two sums $k=(k_1,k_2)$ runs over all couples satisfying the indicated inequality.) 

Using the inequalities \eqref{35} it follows that 
\begin{multline}\label{38}
\int_0^T\int_a^b\int_0^{\pi}\abs{u'(t,x_1,x_2)}^2\ dx_2\ dx_1\ dt\\
\ge \left( T-\frac{32\pi^2}{T} \right)m_{a,b}\sum_{k_2=1}^{\infty}\sum_{k_1=k_2}^{\infty}\abs{k}^2\left( \abs{a_k}^2 +\abs{b_k}^2 \right)\\
-\frac{16\pi^3}{T}\sum_{k_2=1}^{\infty}\sum_{k_1=1}^{k_2-1}\abs{k}^2\left( \abs{a_k}^2 +\abs{b_k}^2 \right).
\end{multline}

By symmetry we also have 
\begin{multline}\label{39}
\int_0^T\int_0^{\pi}\int_c^d\abs{u'(t,x_1,x_2)}^2\ dx_2\ dx_1\ dt\\
\ge \left( T-\frac{32\pi^2}{T} \right)m_{c,d}\sum_{k_1=1}^{\infty}\sum_{k_2=k_1}^{\infty}\abs{k}^2\left( \abs{a_k}^2 +\abs{b_k}^2 \right)\\
-\frac{16\pi^3}{T}\sum_{k_1=1}^{\infty}\sum_{k_2=1}^{k_1-1}\abs{k}^2\left( \abs{a_k}^2 +\abs{b_k}^2 \right).
\end{multline}
In case $T^2\ge 32\pi^2$ this implies the following, slightly weaker inequality:
\begin{multline}\label{310}
\int_0^T\int_0^{\pi}\int_c^d\abs{u'(t,x_1,x_2)}^2\ dx_2\ dx_1\ dt\\
\ge \left( T-\frac{32\pi^2}{T} \right)m_{c,d}\sum_{k_1=1}^{\infty}\sum_{k_2=k_1+1}^{\infty}\abs{k}^2\left( \abs{a_k}^2 +\abs{b_k}^2 \right)\\
-\frac{16\pi^3}{T}\sum_{k_1=1}^{\infty}\sum_{k_2=1}^{k_1}\abs{k}^2\left( \abs{a_k}^2 +\abs{b_k}^2 \right).
\end{multline}

Now adding \eqref{38} to \eqref{310} we obtain that 
\begin{multline}\label{311}
\int_0^T\int_a^b\int_0^{\pi}\abs{u'(t,x_1,x_2)}^2\ dx_2\ dx_1\ dt\\
+\int_0^T\int_0^{\pi}\int_c^d\abs{u'(t,x_1,x_2)}^2\ dx_2\ dx_1\ dt\\
\ge \left( Tm_{a,b}-\frac{32\pi^2}{T}m_{a,b} -\frac{16\pi^3}{T}\right)\sum_{k_2=1}^{\infty}\sum_{k_1=k_2}^{\infty}\abs{k}^2\left( \abs{a_k}^2 +\abs{b_k}^2 \right)\\
+ \left( Tm_{c,d} -\frac{32\pi^2}{T}m_{c,d} -\frac{16\pi^3}{T}\right)\sum_{k_2=1}^{\infty}\sum_{k_1=1}^{k_2-1}\abs{k}^2\left( \abs{a_k}^2 +\abs{b_k}^2 \right).
\end{multline}
Setting
\begin{equation}\label{312}
m:=\min\set{m_{a,b},m_{c,d}}
\end{equation}
we conclude that
\begin{multline}\label{311new}
\int_0^T\int_a^b\int_0^{\pi}\abs{u'(t,x_1,x_2)}^2\ dx_2\ dx_1\ dt\\
+\int_0^T\int_0^{\pi}\int_c^d\abs{u'(t,x_1,x_2)}^2\ dx_2\ dx_1\ dt\\
\ge \left( Tm -\frac{32\pi^2}{T}m  -\frac{16\pi^3}{T}\right)\sum_{k_1=1}^{\infty}\sum_{k_2=1}^{\infty}\abs{k}^2\left( \abs{a_k}^2 +\abs{b_k}^2 \right).
\end{multline}

If $T>0$ is chosen such that
\begin{equation}\label{313}
T^2> 32\pi^2+\frac{16\pi^3}{m},
\end{equation}
then in view of \eqref{37} we get \eqref{17} with
\begin{equation}\label{314}
c=\frac{2m}{\pi^2T}\left(T^2- 32\pi^2+\frac{16\pi^3}{m}\right)>0.\qedhere
\end{equation} 
\end{proof}

\begin{proof}[Proof of the estimate \eqref{18}]
Proceeding as in the preceding proof, we obtain for $T>4\sqrt{2}\pi$ the following inequalities:
\begin{align*}
\int_0^T\int_0^{\pi}&\abs{u_{x_2}(t,x_1,0)}^2\ dx_1\ dt\\
&=\int_0^T\int_0^{\pi}\abs{\sum_{k_1=1}^{\infty}\sum_{k_2=1}^{\infty}k_2\left( a_ke^{i\abs{k}t}+b_{k}e^{-i\abs{k}t}\right)\sin (k_1x_1)}^2\ dx_1\ dt\\
&=\frac{\pi}{2}\sum_{k_1=1}^{\infty}\int_0^T\abs{\sum_{k_2=1}^{\infty}\left( k_2a_ke^{i\abs{k}t}+k_2b_{k}e^{-i\abs{k}t}\right)}^2\ dt\\
&\ge \frac{\pi}{2}\sum_{k_1=1}^{\infty}\left( \frac{2T}{\pi}-\frac{64\pi}{T}\right)\sum_{k_2\ge k_1}k_2^2\left( \abs{a_k}^2 +\abs{b_k}^2 \right)\\
&\qquad\qquad -\frac{\pi}{2}\sum_{k_1=1}^{\infty}\frac{64\pi}{T}\sum_{k_2<k_1}k_2^2\left( \abs{a_k}^2 +\abs{b_k}^2 \right) \\
&=\left( T-\frac{32\pi^2}{T}\right)\sum_{k_1=1}^{\infty}\sum_{k_2\ge k_1}k_2^2\left( \abs{a_k}^2 +\abs{b_k}^2 \right) \\
&\qquad\qquad -\frac{32\pi^2}{T}\sum_{k_1=1}^{\infty}\sum_{k_2<k_1}k_2^2\left( \abs{a_k}^2 +\abs{b_k}^2 \right) \\
&\ge \left( T-\frac{32\pi^2}{T}\right)\sum_{k_1=1}^{\infty}\sum_{k_2> k_1}k_2^2\left( \abs{a_k}^2 +\abs{b_k}^2 \right) \\
&\qquad\qquad -\frac{32\pi^2}{T}\sum_{k_1=1}^{\infty}\sum_{k_2\le k_1}k_2^2\left( \abs{a_k}^2 +\abs{b_k}^2 \right) .
\end{align*}
Since $k_2^2\ge\abs{k}^2/2$ if $k_1\ge k_2$, and $k_2^2\le\abs{k}^2/2$ if $k_1< k_2$, it follows that 
\begin{multline}\label{315}
\int_0^T\int_0^{\pi}\abs{u_{x_2}(t,x_1,0)}^2\ dx_1\ dt\\
\ge \left( \frac{T}{2}-\frac{16\pi^2}{T}\sum_{k_1=1}^{\infty}\right)\sum_{k_2> k_1}\abs{k}^2\left( \abs{a_k}^2 +\abs{b_k}^2 \right) \\
-\frac{32\pi^2}{T}\sum_{k_1=1}^{\infty}\sum_{k_2\le k_1}\abs{k}^2\left( \abs{a_k}^2 +\abs{b_k}^2 \right) .
\end{multline}

Adding \eqref{38} to \eqref{315} we obtain that
\begin{multline*}
\int_0^T\int_a^b\int_0^{\pi}\abs{u'(t,x_1,x_2)}^2\ dx_2\ dx_1\ dt
+\int_0^T\int_0^{\pi}\abs{\partial_{\nu}(t,x_1,0)}^2 \ dx_1\ dt\\
\ge \left( \frac{T}{2}-\frac{16\pi^2}{T}-\frac{16\pi^3}{T}\right)\sum_{k_1=1}^{\infty}\sum_{k_2> k_1}\abs{k}^2\left( \abs{a_k}^2 +\abs{b_k}^2 \right) \\
+\left( Tm_{a,b}-\frac{32\pi^2m_{a,b}}{T}-\frac{32\pi^2}{T}\right)\sum_{k_1=1}^{\infty}\sum_{k_2\le k_1}\abs{k}^2\left( \abs{a_k}^2 +\abs{b_k}^2 \right).
\end{multline*}
If $T>0$ satisfies
\begin{equation*}
T^2>\max\set{32\pi^2+32\pi^3,32\pi^2+\frac{32\pi^2}{m_{a,b}}},
\end{equation*} 
then in view of \eqref{37} the estimate \eqref{18} follows with 
\begin{equation*}
c:=\frac{1}{\pi^2T}\min\set{T^2-32\pi^2-32\pi^3,2T^2m_{a,b}-64\pi^2m_{a,b}-64\pi^2}.\qedhere
\end{equation*} 
\end{proof}

\section{Proof of Theorem \ref{t17}}\label{s4}

We start with a lemma. In order to state it in a convenient way, we extend  $f\in L^1(0,\pi)$ to a $2\pi$-periodic odd function:

\begin{lemma}\label{l41}
A locally integrable, $2\pi$-periodic odd function $f$ satisfies for some positive integer $p$ the relations
\begin{equation}\label{41}
\int_0^{\pi}f(x)\sin(pmx)\ dx=0,\quad\text{for all}\quad  m=1,2,\ldots
\end{equation}
if and only if it is $p$-symmetric, i.e.,
\begin{equation}\label{42}
\sum_{k=1}^pf\left( t+\frac{2k\pi}{p}\right)=0
\end{equation}
for almost every $t\in\RR$.
\end{lemma}

\begin{proof}
Since $f$ is odd, we have for each positive integer $m$ the following equality:
\begin{equation}\label{43}
 \int_0^{\pi}f(x)\sin(pmx)\ dx=\frac{1}{2i}\int_{-\pi}^{\pi}f(y)e^{impy}\ dy.
\end{equation} 
Indeed, we have 
\begin{align*}
 \int_0^{\pi}f(x)\sin(pmx)\ dx
&=\frac{1}{2i}\int_{0}^{\pi}f(x)e^{impx}\ dx-\frac{1}{2i}\int_{0}^{\pi}f(x)e^{-impx}\ dx\\
&=\frac{1}{2i}\int_{0}^{\pi}f(x)e^{impx}\ dx+\frac{1}{2i}\int_{0}^{-\pi}f(-y)e^{impy}\ dy\\
&=\frac{1}{2i}\int_{0}^{\pi}f(x)e^{impx}\ dx+\frac{1}{2i}\int^{0}_{-\pi}-f(-y)e^{impy}\ dy;
\end{align*}
since $-f(-y)=f(y)$, the equality \eqref{43} follows.

It follows that the conditions \eqref{41} are equivalent to the conditions 
\begin{equation*}
\int_{-\pi}^{\pi}f(y)e^{impy}\ dy=0\quad\text{for all}\quad  m=1,2,\ldots
\end{equation*}

Since $f$ is odd, this is equivalent to 
\begin{equation*}
\int_{-\pi}^{\pi}f(y)e^{impy}\ dy=0\quad\text{for all integers}\quad  m;
\end{equation*}
since, furthermore, $f$ is $2\pi$-periodic,  \eqref{41} is also equivalent to 
\begin{equation}\label{44}
 \int_{0}^{2\pi}f(y)e^{impy}\ dy=0\quad\text{for all integers}\quad  m.
\end{equation} 

Next we establish for each fixed integer $m$ the identity 
\begin{equation}\label{45}
 \int_{0}^{2\pi}f(y)e^{impy}\ dy
=\frac{1}{p}\int_0^{2\pi}\sum_{k=0}^{p-1}f\left(\frac{u+2k\pi}{p}\right)e^{imu}\ du.
\end{equation} 
Indeed, we have 
\begin{align*}
  \int_{0}^{2\pi}f(y)e^{impy}\ dy
&=\frac{1}{p}\int_0^{2p\pi}f\left(\frac{z}{p}\right)e^{imz}\ dz\\
&=\frac{1}{p}\sum_{k=0}^{p-1}\int_0^{2\pi}f\left(\frac{u+2k\pi}{p}\right)e^{imu}\ du.
\end{align*}
In view of \eqref{45} the condition \eqref{44} is equivalent to the equality 
\begin{equation*}
 \sum_{k=0}^{p-1}f\left(t+k\frac{2\pi}{p}\right)=0
\end{equation*}
for almost every $t\in (0,2\pi/p)$ or equivalently for almost every $t\in\RR$.
\end{proof}

\begin{remark}
 For $p=2$ the condition \eqref{41} means that $f$ is even with respect to the middle-point of the interval $(0,\pi)$.
\end{remark}

Turning to the proof of the theorem, we will use the results of the preceding section.

\begin{proof}[Proof of the estimate \eqref{19}]
Proceeding in the usual way we have
\begin{align*}
\int_0^T\int_0^{\pi}&\abs{u'(t,\alpha ,x_2)}^2\ dx_2\ dt\\
&=\int_0^T\int_0^{\pi}\abs{\sum_{k_1=1}^{\infty}\sum_{k_2=1}^{\infty}-i\abs{k}\left( a_ke^{i\abs{k}t}-b_{k}e^{-i\abs{k}t}\right)\sin (k_1\alpha)\sin (k_2x_2)}^2\ dx_2\ dt\\
&=\frac{\pi}{2}\sum_{k_2=1}^{\infty}\int_0^T\abs{\sum_{k_1=1}^{\infty}\abs{k}\left( a_ke^{i\abs{k}t}-b_{k}e^{-i\abs{k}t}\right)\sin (k_1\alpha)}^2\ dt\\
&\ge \frac{\pi}{2}\sum_{k_2=1}^{\infty}\left( \frac{2T}{\pi}-\frac{64\pi}{T}\right)\sum_{k_1\ge k_2}\abs{k}^2\left( \abs{a_k}^2 +\abs{b_k}^2 \right)\sin^2 (k_1\alpha)\\
&\qquad\qquad -\frac{\pi}{2}\sum_{k_2=1}^{\infty}\frac{64\pi}{T}\sum_{k_1<k_2}\abs{k}^2\left( \abs{a_k}^2 +\abs{b_k}^2 \right)\sin^2 (k_1\alpha).
\end{align*}

Since the initial data are $p$-symmetric in $x_1$, it follows from the preceding lemma that $a_k=b_k=0$ whenever $\sin^2 (k_1\alpha)=0$. Therefore, denoting by $m_p$ and $M_p$ the minimum and maximum of the nonzero values of $\sin^2 (k_1\alpha)$ when $k_1$ runs over integers (we have to compare only $p-1$ values because of the $p$-periodicity), we conclude that 
\begin{multline}\label{46}
\int_0^T\int_0^{\pi}\abs{u'(t,\alpha ,x_2)}^2\ dx_2\ dt\\
\ge \sum_{k_2=1}^{\infty}\left( T-\frac{32\pi^2}{T}\right)m_p\sum_{k_1\ge k_2}\abs{k}^2\left( \abs{a_k}^2 +\abs{b_k}^2 \right) \\
-\sum_{k_2=1}^{\infty}\frac{32\pi^2}{T}M_p\sum_{k_1<k_2}\abs{k}^2\left( \abs{a_k}^2 +\abs{b_k}^2 \right)  .
\end{multline}

Adding \eqref{46} to \eqref{39} we get
\begin{multline*}
\int_0^T\int_0^{\pi} \abs{u'(t,\alpha ,x_2)}^2\ dx_2 \ dt
+\int_0^T\int_0^{\pi}\int_c^d\abs{u'(t,x_1,x_2)}^2\ dx_2\ dx_1\ dt\\
\ge \left( Tm_p-\frac{32\pi^2m_p}{T}-\frac{16\pi^3}{T}\right)\sum_{k_1\ge k_2}\abs{k}^2\left( \abs{a_k}^2 +\abs{b_k}^2 \right) \\
+\left( Tm_{c,d}-\frac{32\pi^2m_{c,d}}{T}-\frac{32\pi^2M_p}{T}\right)\sum_{k_1<k_2}\abs{k}^2\left( \abs{a_k}^2 +\abs{b_k}^2 \right)  .
\end{multline*}

As usual, using \eqref{37} we conclude that if
\begin{equation*}
T^2>\max\set{32\pi^2+\frac{16\pi^3}{m_p},32\pi^2+\frac{32\pi^2M_p}{m_{c,d}}},
\end{equation*} 
then the estimate \eqref{19} holds with 
\begin{equation*}
c:=\frac{2}{\pi^2T}\min\set{T^2m_p-32\pi^2m_p-16\pi^3,T^2m_{c,d}-32\pi^2m_{c,d}-32\pi^2M_p}.\qedhere
\end{equation*} 
\end{proof}

\begin{proof}[Proof of the estimate \eqref{110}]
Taking the sum of the already established inequalities \eqref{315} and \eqref{46} we have 
\begin{multline*}
\int_0^T\int_0^{\pi} \abs{u'(t,\alpha ,x_2)}^2\ dx_2 \ dt
+\int_0^T\int_0^{\pi}\abs{\partial_{\nu}(t,x_1,0)}^2 \ dx_1\ dt\\
\ge \left( \frac{T}{2}-\frac{16\pi^2}{T}-\frac{32\pi^2M_p}{T}\right)\sum_{k_2> k_1}\abs{k}^2\left( \abs{a_k}^2 +\abs{b_k}^2 \right) \\
\left( Tm_p-\frac{32\pi^2m_p}{T}-\frac{32\pi^2}{T}\right)\sum_{k_2\le k_1}\abs{k}^2\left( \abs{a_k}^2 +\abs{b_k}^2 \right).
\end{multline*}
We conclude that if 
\begin{equation*}
T^2>32\pi^2\max\set{1+2M_p,1+\frac{1}{m_p}},
\end{equation*} 
then the estimate \eqref{110} holds with 
\begin{equation*}
c:=\frac{1}{\pi^2T}\min\set{ T^2 - 32\pi^2 - 64\pi^2M_p ,2T^2m_p- 64\pi^2m_p-64\pi^2}.\qedhere
\end{equation*} 
\end{proof}

\section{Proof of Theorem \ref{t18}}\label{s5}

\begin{proof}[Proof of Theorem \ref{t18}] Since the initial data are doubly symmetric, the solutions satisfy not only the inequality \eqref{46} of the preceding section, but also the analogous inequality obtained by exchanging the role of $x_1$ and $x_2$:
\begin{multline*}
\int_0^T\int_0^{\pi}\abs{u'(t,x_1,\beta)}^2\ dx_1\ dt\\
\ge \sum_{k_1=1}^{\infty}\left( T-\frac{32\pi^2}{T}\right)m_q\sum_{k_2\ge k_1}\abs{k}^2\left( \abs{a_k}^2 +\abs{b_k}^2 \right) \\
-\sum_{k_1=1}^{\infty}\frac{32\pi^2M_q}{T}\sum_{k_2<k_1}\abs{k}^2\left( \abs{a_k}^2 +\abs{b_k}^2 \right),
\end{multline*}
whence
\begin{multline}\label{51}
\int_0^T\int_0^{\pi}\abs{u'(t,x_1,\beta)}^2\ dx_1\ dt\\
\ge \sum_{k_1=1}^{\infty}\left( T-\frac{32\pi^2}{T}\right)m_q\sum_{k_2> k_1}\abs{k}^2\left( \abs{a_k}^2 +\abs{b_k}^2 \right) \\
-\sum_{k_1=1}^{\infty}\frac{32\pi^2M_q}{T}\sum_{k_2\le k_1}\abs{k}^2\left( \abs{a_k}^2 +\abs{b_k}^2 \right).
\end{multline}
Adding \eqref{46} and \eqref{51} and setting 
\begin{equation*}
M_{p,q}:=\max\set{m_p+M_q,m_q+M_p}
\end{equation*}
for brevity, we obtain that
\begin{multline*}
\int_0^T\int_0^{\pi}\abs{u'(t,\alpha,x_2)}^2\ dx_2\ dt
+\int_0^T\int_0^{\pi}\abs{u'(t,x_1,\beta)}^2\ dx_1\ dt\\
\ge \left( T-\frac{32\pi^2M_{p,q}}{T}\right)\sum_{k_1=1}^{\infty}\sum_{k_2=1}^{\infty}\abs{k}^2\left( \abs{a_k}^2 +\abs{b_k}^2 \right).
\end{multline*}
Hence for $T^2>32\pi^2M_{p,q}$ the estimate \eqref{111} holds with 
\begin{equation*}
c:=\frac{2}{\pi^2T}\left( T^2-32\pi^2M_{p,q}\right).\qedhere
\end{equation*} 
\end{proof}

\begin{remark}
The case $\alpha=\beta=\pi/2$ is particularly simple because $m_p=M_p=m_q=M_q=1$ and therefore $M_{p,q}=2$: if $T>8\pi$ then the estimate \eqref{111} holds with 
\begin{equation*}
c:=\frac{2\left( T^2-64\pi^2\right)}{\pi^2T}.
\end{equation*} 
\end{remark}


\begin{thebibliography}{29}


\bibitem{BaiKomLor103} C. Baiocchi, V. Komornik and P.
Loreti,
{\em Ingham type theorems and  applications to control theory},
Bol. Un. Mat. Ital. B (8) 2 (1999), no. 1, 33–63.

\bibitem{BaiKomLor111} C. Baiocchi, V. Komornik and P.
Loreti,
{\em Ingham–Beurling type theorems with weakened gap conditions},
Acta Math. Hungar. 97 (2002), 1–2, 55–95.

\bibitem{BarLebRau1992} C. Bardos,  G. Lebeau, and   J. Rauch,
{\em Sharp sufficient conditions for the observation, control and
stabilization of waves from the boundary},
SIAM J. Control Optim. 30 (1992), 1024--1065.

\bibitem{Cas} J. W. S. Cassels,
{\em An Introduction to Diophantine Approximation},
Cambridge Tracts in Mathematics and Mathematical Physics,
No. 45. Cambridge University Press, New York, 1957.

\bibitem{Har1988} A. Haraux,
\emph{On a completion problem in the theory of distributed control of wave equations}, 
Nonlinear partial differential equations and their applications. 
Collège de France Seminar, Vol. X (Paris, 1987–1988), 241--271, 
Pitman Res. Notes Math. Ser., 220, Longman Sci. Tech., Harlow, 1991. 

\bibitem{Har1989} A. Haraux, 
{\em Séries lacunaires et contrôle semi-interne des vibrations d'une plaque rectangulaire}, 
J. Math. Pures Appl. 68 (1989), 457–465.

\bibitem{Ho1986} L. F. Ho, {\em Observabilité frontière de
l'équation des ondes}, C. R. Acad. Sci. Paris Sér. I Math. 302 (1986),
443--446.

\bibitem{Ing1936} A. E. Ingham,
{\em Some trigonometrical inequalities with
applications in the
theory of series}, Math. Z. 41 (1936), 367–379.

\bibitem{Jaf1988} S. Jaffard, {\em Contrôle interne exact des  vibrations d'une plaque carrée},
C. R. Acad. Sci. Paris Sér. I Math. 307 (1988), 759--762.

\bibitem{Jaf1990} S. Jaffard, {\em Contrôle interne exact des vibrations d'une plaque rectangulaire}, 
Portugalia Math. 47 (1990), 423--429.

\bibitem{KomLorbook} V. Komornik,  P. Loreti,  {\em Fourier Series in
Control Theory}, Springer-Verlag, New York, 2005.

\bibitem{KomLor2011-152} V. Komornik,  P. Loreti,  
\emph{Multiple-point internal observability of membranes and plates}, 
Appl. Anal. 90 (2011), 10, 1545–1555.

\bibitem{KomMia2013} 
V. Komornik, B. Miara,
\emph{Internal observability of rectangular membranes,}
submitted.

\bibitem{Lio1988} J.-L. Lions, {\em Exact controllability,
stabilizability, and perturbations for distributed systems}, Siam Rev.
30 (1988), 1--68.

\bibitem{Lio1988b} J.-L. Lions, {\em Contrôlabilité exacte et
stabilisation  de systèmes distribués I-II}, Masson, Paris,
1988.

\bibitem{Lor2004} P. Loreti,
{\em On some gap theorems},
Proceedings of the 11th Meeting of EWM, CWI Tract (2005).

\bibitem{LorMeh} P. Loreti and M. Mehrenberger,
\emph{An Ingham type proof for a two-grid observability theorem,}
ESAIM Control Optim. Calc. Var. 14 (2008), no. 3, 604--631. 

\bibitem{LorVal} P. Loreti and V. Valente,
{\em Partial exact controllability for spherical  membranes},
SIAM J. Control Optim. 35 (1997), 641--653.

\bibitem{Meh2009}
M. Mehrenberger,
\emph{An Ingham type proof for the boundary observability of a N–d wave equation,}
C. R. Math. Acad. Sci. Paris 347 (2009), no. 1-2, 63–68.

\bibitem{TenTuc} G. Tenenbaum, M. Tucsnak,
\emph{Fast and strongly localized observation for the Schrödinger equation,}
Trans. Amer. Math. Soc. 361 (2009), no. 2, 951–977. 

\end{thebibliography}
\end{document}